\documentclass[11pt]{amsart}
\usepackage{amssymb}
\usepackage{xcolor}
\usepackage{mathrsfs}
\usepackage{amsfonts}
\usepackage{amsmath}
\usepackage{mathtools}
\usepackage{tikz}

\numberwithin{equation}{section}
\allowdisplaybreaks[4]
\newtheorem{Theorem}{Theorem}[section]

\newtheorem{Proposition}[Theorem]{Proposition}

\newtheorem{Lemma}[Theorem]{Lemma}

\newtheorem*{theorema*}{Theorem A}
\unless\ifcsname proof\endcsname

\fi

\def\supp{{\rm supp}\ }

\def\bbZ{\mathbb{Z}}

\def\bbR{\mathbb{R}}

\begin{document}

\title[Two weight inequality]{Two weight norm inequalities for the bilinear fractional integrals}

\thanks{This work was partially supported  by the
National Natural Science Foundation of China(11371200) and the Research Fund for the Doctoral Program
of Higher Education (20120031110023).}

\subjclass[2010]{42B20, 42B25}

\author{Kangwei Li}
\address{School of Mathematical Sciences and LPMC,  Nankai University,
      Tianjin~300071, China}
\email{likangwei9@mail.nankai.edu.cn}

\author{Wenchang Sun}

\address{School of Mathematical Sciences and LPMC,  Nankai University,
      Tianjin~300071, China}
\email{sunwch@nankai.edu.cn}

\begin{abstract}
In this paper, we give a characterization of the two weight strong and weak type norm inequalities for the bilinear fractional integrals. Namely, we give the characterization of the following inequalities,
\[
 \|\mathcal I_\alpha (f_1\sigma_1, f_2\sigma_2)\|_{L^q(w)} \le  \mathscr N \prod_{i=1}^2\|f_i\|_{L^{p_i}(\sigma_i)}
\]
and
\[
 \|\mathcal I_\alpha (f_1\sigma_1, f_2\sigma_2)\|_{L^{q,\infty}(w)} \le   \mathscr N_{\textup{weak}} \prod_{i=1}^2\|f_i\|_{L^{p_i}(\sigma_i)},
\]
when $q\ge p_1, p_2>1$ and $p_1+p_2\ge p_1p_2$.
\end{abstract}

\keywords{
Bilinear fractional integral; two weight inequality.}
\maketitle
\section{Introduction and Main Results}
By a weight we mean a positive locally finite Borel measure on $\bbR^n$. We begin with the definition of the bilinear fractional integral
 $\mathcal I_\alpha(\cdot \sigma_1, \cdot \sigma_2)$.  For suitable functions $f_1$ and $f_2$,  define
\[
  \mathcal I_\alpha (f_1\sigma_1, f_2\sigma_2)(x)=\int_{\bbR^{2n}}\frac{f_1(y_1)f_2(y_2)}{(|x-y_1|+|x-y_2|)^{2n-\alpha}}d\sigma_1 d\sigma_2.
\]
Observe that
\[
  |x-y_1|+|x-y_2|\simeq |y_1-x|+|y_1-y_2|\simeq |y_2-x|+|y_2-y_1|.
\]
We know that $\mathcal I_\alpha$ is equivalent to its duals $\mathcal I_\alpha^{1,*}$ and $\mathcal I_\alpha^{2,*}$.

In this paper, we concern the following strong type weighted norm inequality,
\begin{equation}\label{eq:norm}
 \|\mathcal I_\alpha (f_1\sigma_1, f_2\sigma_2)\|_{L^q(w)}\le  \mathscr N \prod_{i=1}^2\|f_i\|_{L^{p_i}(\sigma_i)},
\end{equation}
and the weak type weighted norm inequality,
\begin{equation}\label{eq:weak}
 \|\mathcal I_\alpha (f_1\sigma_1, f_2\sigma_2)\|_{L^{q,\infty}(w)}\le  \mathscr N_{\textup{weak}} \prod_{i=1}^2\|f_i\|_{L^{p_i}(\sigma_i)},
\end{equation}
where $\mathscr N$ and
$\mathscr N_{\textup{weak}}$ are the best constants such that the above inequalities hold, respectively. We aim to give a characterization of \eqref{eq:norm} and \eqref{eq:weak} using   Sawyer type test conditions.

In the linear case, the characterization of weighted norm inequalities have attracted many authors. For the maximal operators, we refer the readers to the works of
Sawyer \cite{S} and Moen \cite{M1}. For the fractional integrals, we refer the readers to \cite{K, Sawyer,Saw, SWZ}. And for the Calder\'on-Zygmund operators, this problem is referred to as the Nazarov-Treil-Volberg conjecture \cite{V}. This conjecture
has been solved for the Hilbert transform, see the remarkable work of Lacey \cite{L1, L2} and the recent work of Hyt\"onen \cite{Hytonen}.
For the vector Riesz transform, it was partially solved by Sawyer, Shen and Uriarte-Tuero in \cite{SSU1}, where they gave a characterization under the assumption that at least one of the two weights is supported on a line. And in \cite{LW}, Lacey and Wick gave a characterization under the hypotheses that the two weights separately are not concentrated on a
set of codimension one, uniformly over locations and scales.
There is also another approach on this topic, namely, finding
 a minimal sufficient condition of the weights such that the two weight inequality holds. We refer the readers to \cite{ACM, CMP, CP1, CP2, L,Lerner} and  references therein. For other related works, we refer the readers to \cite{LMS,LS,MXL,M,NTV1,NTV3}.

Now the story goes to the multilinear case. In \cite{LS}, we studied the characterization of two weight norm inequalities for the multilinear fractional maximal operators using   Sawyer type test conditions. Recall that the multilinear fractional maximal operators are defined by
\[
   \mathcal{M}_\alpha (\vec f)(x)=\sup_{Q\ni x}\prod_{i=1}^m\frac{1}{|Q|^{1-\alpha/{mn}}}\int_Q |f_i(y_i)|dy_i,
\]
where $0\le\alpha<mn$.

In this paper, we give a characterization of the two weight strong and weak type norm inequalities for the bilinear fractional integrals.
Specifically, we prove the following.
\begin{Theorem}\label{thm:main}
Let $\sigma_1,\sigma_2, w$ be positive locally finite Borel measures and $q\ge p_1, p_2>1$ with $p_1+p_2\ge p_1p_2$. Then \eqref{eq:norm}
holds if and only if the following test conditions hold
\begin{eqnarray*}
\mathcal T&:=&\sup_Q\frac{(\int_Q \mathcal I_\alpha (1_Q\sigma_1, 1_Q\sigma_2)^q dw)^{1/q}}{\sigma_1(Q)^{1/{p_1}}\sigma_2(Q)^{1/{p_2}}}<\infty;\\
\mathcal T_1^*&:=&\sup_Q\frac{(\int_Q \mathcal I_\alpha (1_Q w, 1_Q\sigma_2)^{p_1'} d\sigma_1)^{1/{p_1'}}}{w(Q)^{1/{q'}}\sigma_2(Q)^{1/{p_2}}}<\infty;\\
\mathcal T_2^*&:=&\sup_Q\frac{(\int_Q \mathcal I_\alpha (1_Q\sigma_1, 1_Q w)^{p_2'} d\sigma_2)^{1/{p_2'}}}{\sigma_1(Q)^{1/{p_1}}w(Q)^{1/{q'}}}<\infty.
\end{eqnarray*}
Moreover,  $\mathscr N \simeq \mathcal T+\mathcal T_1^*+\mathcal T_2^*$. And \eqref{eq:weak} holds if and only if $\mathcal T_1^*, \mathcal T_2^*<\infty$.
Moreover, $\mathscr N_{\textup{weak}}\simeq \mathcal T_1^*+\mathcal T_2^*$.
\end{Theorem}

This paper is organized as follows. In Section~\ref{sec:s1}, we reduce the problem to the dyadic bilinear fractional integral and give some preliminary estimates. In Section~\ref{sec:s2}
 and Section~\ref{sec:s3}, we  give a proof for Theorem~\ref{thm:main}.

\section{Preliminaries}\label{sec:s1}
In this section, as in \cite{CM}, we define two dyadic versions of the bilinear fractional integral. We  show that they are equivalent with
the bilinear fractional integral pointwise.
Firstly, we introduce the following result, which can be found in \cite[Proof of Theorem 1.7]{HP}.
\begin{Proposition}\label{prop:p1}
There are $2^n$ dyadic grids $\mathscr{D}_t$, $t\in\{0,1/3\}^n$ such that for any  cube
$Q\subset \bbR^n$ there exists a cube $Q_t \in \mathscr{D}_t$ satisfying
 $Q \subset Q_t$ and $l(Q_t)\le 6l(Q)$, where
 \[
   \mathscr{D}_t:=\{2^{-k}([0,1)^n+m+(-1)^k t): k\in\bbZ, m\in\bbZ^n\},\quad t\in\{0, 1/3\}^n.
 \]
\end{Proposition}
Given a dyadic grid $\mathscr D$, we define the dyadic bilinear fractional integral
\[
  \mathcal I_{\alpha}^{\mathscr D}(f_1\sigma_1, f_2\sigma_2)(x):=\sum_{Q\in\mathscr D}\prod_{i=1}^2\frac{1}{|Q|^{1-\alpha/{2n}}}
  \int_Q f_i d\sigma_i\cdot \chi_Q(x).
\]
Analogue to the argument in \cite{CM}, we have the following result.
\begin{Proposition}\label{prop:dyadic}
Given $0<\alpha<2n$, positive locally finite Borel measures $\sigma_1, \sigma_2$ and non-negative functions $ f_1, f_2$, then for any dyadic grid $\mathscr D$,
\[
  \mathcal I_{\alpha}^{\mathscr D}(f_1\sigma_1, f_2\sigma_2)(x)\lesssim \mathcal I_{\alpha}(f_1\sigma_1, f_2\sigma_2)(x).
\]
Conversely, we have
\[
  \mathcal I_{\alpha}(f_1\sigma_1, f_2\sigma_2)(x)\lesssim \max_{t\in\{0,1/3\}^n} \mathcal I_{\alpha}^{\mathscr D_t}(f_1\sigma_1, f_2\sigma_2)(x).
\]
\end{Proposition}
Notice that with Proposition~\ref{prop:dyadic}, we can get the following
\begin{equation}\label{eq:dyadic}
 \mathcal I_{\alpha}(f_1\sigma_1, f_2\sigma_2)(x)\simeq \sum_{t\in\{0,1/3\}^n} \mathcal I_{\alpha}^{\mathscr D_t}(f_1\sigma_1, f_2\sigma_2)(x).
\end{equation}
\begin{proof}
Fix some $x\in\bbR^n$.
Denote by $\{Q_k\}_{k\in\bbZ}$ the unique sequence in $\mathscr D$ such that $x\in Q_k$ and $l(Q_k)=2^k$. Fix $N\ge 1$. We have
\begin{eqnarray*}
&&\sum_{Q\in\mathscr D\atop 2^{-N}\le l(Q)\le 2^N}\prod_{i=1}^2\frac{1}{|Q|^{1-\frac{\alpha}{2n}}}
  \int_Q f_i d\sigma_i\cdot \chi_Q(x)\\
  &=& \sum_{k=-N}^N \prod_{i=1}^2\frac{1}{|Q_k|^{1-\frac{\alpha}{2n}}}
  \int_{Q_k} f_i d\sigma_i\\
  &=& \sum_{k=-N}^N \frac{1}{|Q_k|^{2-\frac{\alpha}{n}}}
  \iint_{(Q_k\times Q_k)\setminus (Q_{k-1}\times Q_{k-1})} f_1 f_2d\sigma_1d\sigma_2\\
  &&\quad+ \sum_{k=-N}^N \prod_{i=1}^2\frac{1}{|Q_k|^{1-\frac{\alpha}{2n}}}
  \int_{Q_{k-1}} f_i d\sigma_i\\
  &\le & C\sum_{k=-N}^N \iint_{(Q_k\times Q_k)\setminus (Q_{k-1}\times Q_{k-1})}\frac{f_1(y_1)f_2(y_2)}{(|x-y_1|+|x-y_2|)^{2n-\alpha}}d\sigma_1 d\sigma_2\\
  &&\quad+2^{\alpha-2n}\prod_{i=1}^2\frac{  \int_{Q_{-N-1}} f_i d\sigma_i}{|Q_{-N-1}|^{1-\frac{\alpha}{2n}}}
+ 2^{\alpha-2n}\sum_{k=-N}^{N-1} \prod_{i=1}^2\frac{1}{|Q_k|^{1-\frac{\alpha}{2n}}}
  \int_{Q_{k}} f_i d\sigma_i\\
  &\le&C  \iint_{Q_N\times Q_N}\frac{f_1(y_1)f_2(y_2)}{(|x-y_1|+|x-y_2|)^{2n-\alpha}}d\sigma_1 d\sigma_2\\
  &&\quad+ 2^{\alpha-2n}\sum_{Q\in\mathscr D\atop 2^{-N}\le l(Q)\le 2^N}\prod_{i=1}^2\frac{1}{|Q|^{1-\frac{\alpha}{2n}}}
  \int_Q f_i d\sigma_i\cdot \chi_Q(x).
\end{eqnarray*}
Since $\alpha<2n$, by rearranging the terms and letting $N\rightarrow \infty$, we get
\[
  \mathcal I_{\alpha}^{\mathscr D}(f_1\sigma_1, f_2\sigma_2)(x)\lesssim \mathcal I_{\alpha}(f_1\sigma_1, f_2\sigma_2)(x).
\]
For the second inequality, we refer the readers to \cite{LMS, M}. This completes the proof.
\end{proof}

Next, we define a sparse version of $\mathcal I_\alpha^{\mathscr D}$. We call $\mathcal S\subset \mathscr D$ a sparse family if for every $Q\in\mathcal S$,
\[
  \bigg|\bigcup_{Q'\subsetneq Q\atop Q'\in\mathcal S}Q'\bigg|\le \frac 1 2|Q|.
\]
Now we can define the sparse dyadic bilinear fractional integral by
\[
   \mathcal I_{\alpha}^{\mathcal S}(f_1\sigma_1, f_2\sigma_2)(x):=\sum_{Q\in\mathcal S}\prod_{i=1}^2\frac{1}{|Q|^{1-\alpha/{2n}}}
  \int_Q f_i d\sigma_i\cdot \chi_Q(x).
\]
Next we show that
$\mathcal I_{\alpha}^{\mathcal S}$ and $ \mathcal I_{\alpha}^{\mathscr D}$ are equivalent in some sense.

\begin{Proposition}\label{prop:sparse}
Given $0<\alpha<2n$, positive locally finite Borel measures $\sigma_1, \sigma_2$ and bounded, non-negative, compactly supported functions $f_1, f_2$, then for any dyadic grid $\mathscr D$, there exists a sparse family $\mathcal S\subset \mathscr D$ such that
\[
  \mathcal I_{\alpha}^{\mathscr D}(f_1\sigma_1, f_2\sigma_2)(x)\lesssim \mathcal I_{\alpha}^{\mathcal S}(f_1\sigma_1, f_2\sigma_2)(x).
\]
\end{Proposition}
Notice that $\mathcal S$ is a subfamily of $\mathscr D$.  So we have
\[
  \mathcal I_{\alpha}^{\mathscr D}(f_1\sigma_1, f_2\sigma_2)(x)\simeq \mathcal I_{\alpha}^{\mathcal S}(f_1\sigma_1, f_2\sigma_2)(x).
\]
The proof of Proposition~\ref{prop:sparse} is not essentially different from the linear case, which can be found in \cite{CM,LMPT, Perez, SW}.
\begin{proof}
Let $a=2^{2(n+1)}$. We split $\mathscr D$ to the following subfamilies,
\[
  \mathcal P_k= \{Q\in \mathscr D: a^k<\prod_{i=1}^2 \frac{1}{|Q|}\int_Q f_i d\sigma_i\le a^{k+1}\}.
\]
Then for every $Q$ with $\prod_{i=1}^2 \frac{1}{|Q|}\int_Q f_i d\sigma_i\neq 0$, there is a unique $k$ such that $Q\in \mathcal P_k$.
Therefore, we can write
\begin{eqnarray*}
 \mathcal I_{\alpha}^{\mathscr D}(f_1\sigma_1, f_2\sigma_2)(x)&=&\sum_k\sum_{Q\in\mathcal P_k}\prod_{i=1}^2\frac{1}{|Q|^{1-\frac{\alpha}{2n}}}
  \int_Q f_i d\sigma_i\cdot \chi_Q(x)\\
  &\le& \sum_k a^{k+1}\sum_{Q\in\mathcal P_k} |Q|^{\frac{\alpha}{n}}\chi_Q(x).
\end{eqnarray*}
Denote by $\mathcal S_k$ the collection of maximal dyadic cubes $P\in\mathscr D$ such that
\[
  \prod_{i=1}^2 \frac{1}{|P|}\int_P f_i d\sigma_i>a^k.
\]
Since $\sigma_1$ and $ \sigma_2$ are locally finite and $f_1$ and $ f_2$ are bounded and compactly supported, such a collection exists. Notice that the cubes in $\mathcal S_k$
are pairwise disjoint. Set $\mathcal S=\cup_k \mathcal S_k$. We have
\begin{eqnarray*}
\mathcal I_{\alpha}^{\mathscr D}(f_1\sigma_1, f_2\sigma_2)(x)
 &\le& \sum_k a^{k+1}\sum_{P\in\mathcal S_k}\sum_{Q\in\mathcal P_k\atop Q\subset P} |Q|^{\frac{\alpha}{n}}\chi_Q(x)\\
 &\lesssim& \sum_k\sum_{P\in\mathcal S_k}\bigg(\prod_{i=1}^2\frac{1}{|P|}
  \int_P f_i d\sigma_i\bigg)\sum_{r=0}^\infty \sum_{Q\in\mathcal P_k, Q\subset P\atop l(Q)=2^{-r}l(P)} |Q|^{\frac{\alpha}{n}}\chi_Q(x)\\
  &=& \sum_{S\in\mathcal S}\bigg(\prod_{i=1}^2\frac{1}{|S|}
  \int_S f_i d\sigma_i\bigg)\sum_{r=0}^\infty\sum_{k: S\in \mathcal S_k}\sum_{Q\in\mathcal P_k, Q\subset S\atop l(Q)=2^{-r}l(S)} |Q|^{\frac{\alpha}{n}}\chi_Q(x)\\
  &\le& \sum_{S\in\mathcal S}\bigg(\prod_{i=1}^2\frac{1}{|S|}
  \int_S f_i d\sigma_i\bigg)\sum_{r=0}^\infty\sum_{Q\in\mathscr D, Q\subset S\atop l(Q)=2^{-r}l(S)} |Q|^{\frac{\alpha}{n}}\chi_Q(x)\\
  &\lesssim& \mathcal I_{\alpha}^{\mathcal S}(f_1\sigma_1, f_2\sigma_2)(x).
\end{eqnarray*}
It remains to demonstrate that $\mathcal S$ is sparse. In fact, fix some $P\in\mathcal S_k$. Let $\{P_\lambda\}_{\lambda\in \Lambda}$ be the collection of the maximal dyadic cubes in $\mathcal S$ which
are strictly contained in $P$. Then for any $\lambda\in \Lambda$,
\[
   \prod_{i=1}^2 \frac{1}{|P_\lambda|}\int_{P_\lambda} f_i d\sigma_i>a^{k+1}.
\]
It follows that
\begin{eqnarray*}
\sum_\lambda |P_\lambda|&\le& a^{-(k+1)/2}\sum_{\lambda}\bigg(\prod_{i=1}^2\int_{P_\lambda} f_i d\sigma_i\bigg)^{1/2}\\
&\le& a^{-(k+1)/2} \prod_{i=1}^2\bigg(\sum_{\lambda}\int_{P_\lambda} f_i d\sigma_i\bigg)^{1/2}\\
&\le& a^{-(k+1)/2} \prod_{i=1}^2 \bigg(\int_{P} f_i d\sigma_i\bigg)^{1/2}\\
&\le& \frac 1 2 |P|,
\end{eqnarray*}
where in the last step  we use the fact that
\[
  \prod_{i=1}^2 \frac{1}{|P|}\int_P f_i d\sigma_i\le 2^{2n} \prod_{i=1}^2 \frac{1}{|\hat{P}|}\int_{\hat P} f_i d\sigma_i \le 2^{2n} a^k,
\]
thanks to the maximal property.  Recall that $\hat P$ denotes the father cube of $P$.

\end{proof}
Now we reduce the problem to show the following result.
\begin{Theorem}\label{thm:dyadic}
Let $\mathscr D$ be a dyadic grid and $\mathcal S\subset \mathscr D$ be a sparse family.
Suppose that $\sigma_1,\sigma_2, w$ are positive Borel measures and $q\ge p_1, p_2>1$ with $p_1+p_2\ge p_1p_2$. Then
\begin{equation}\label{eq:dnorm}
\|\mathcal I_\alpha^{\mathcal S} (f_1\sigma_1, f_2\sigma_2)\|_{L^q(w)}\le  \mathcal N \prod_{i=1}^2\|f_i\|_{L^{p_i}(\sigma_i)}
\end{equation}
holds if and only if the following test conditions hold
\begin{eqnarray*}
\mathcal T^{\mathcal S}&:=&\sup_{Q\in\mathscr D}\frac{(\int_Q \mathcal I_\alpha^{\mathcal S} (1_Q\sigma_1, 1_Q\sigma_2)^q dw)^{1/q}}{\sigma_1(Q)^{1/{p_1}}\sigma_2(Q)^{1/{p_2}}}<\infty;\\
\mathcal T_1^{\mathcal S, *}&:=&\sup_{Q\in\mathscr D}\frac{(\int_Q \mathcal I_\alpha^{\mathcal S} (1_Q w, 1_Q\sigma_2)^{p_1'} d\sigma_1)^{1/{p_1'}}}{w(Q)^{1/{q'}}\sigma_2(Q)^{1/{p_2}}}<\infty;\\
\mathcal T_2^{\mathcal S, *}&:=&\sup_{Q\in\mathscr D}\frac{(\int_Q \mathcal I_\alpha^{\mathcal S} (1_Q\sigma_1, 1_Q w)^{p_2'} d\sigma_2)^{1/{p_2'}}}{\sigma_1(Q)^{1/{p_1}}w(Q)^{1/{q'}}}<\infty.
\end{eqnarray*}
Moreover, if $\mathcal N$ is the best constant such that \eqref{eq:dnorm} holds, then $\mathcal N \simeq \mathcal T^{\mathcal S}+\mathcal T_1^{\mathcal S, *}+\mathcal T_2^{\mathcal S, *}$. And
\begin{equation}\label{eq:dweak}
\|\mathcal I_\alpha^{\mathcal S} (f_1\sigma_1, f_2\sigma_2)\|_{L^{q,\infty}(w)}\le  \mathcal N_{\textup{weak}} \prod_{i=1}^2\|f_i\|_{L^{p_i}(\sigma_i)}
\end{equation}
holds if and only if $\mathcal T_1^{\mathcal S, *}, \mathcal T_2^{\mathcal S, *}<\infty$.
Moreover, if $\mathcal N_{\textup{weak}}$ is the best constant such that \eqref{eq:dweak} holds, then $\mathcal N_{\textup{weak}}\simeq \mathcal T_1^{\mathcal S, *}+\mathcal T_2^{\mathcal S, *}$.
\end{Theorem}
In the following, we give some elementary estimates. Assume that $f_1$ and $f_2$ are non-negative.  By the monotone convergence theorem, it suffices to consider
\[
  \mathcal I_\alpha^{\mathcal S(R)} (f_1\sigma_1, f_2\sigma_2)(x)=\sum_{Q\in\mathcal S\atop Q\subset R}\prod_{i=1}^2\frac{1}{|Q|^{1-\alpha/{2n}}}
  \int_Q f_i d\sigma_i\cdot \chi_Q(x),
\]
where $R$ is a cube. In fact, we can further assume that the side-length of any cube in $\mathcal S(R)$ is at least $2^{-m}$.
To avoid miscellaneous subscripts, we  omit the index $m$ in the rest of the paper.
Let
\[
  \Omega_k =\{x: \mathcal I_\alpha^{\mathcal S(R)} (f_1\sigma_1, f_2\sigma_2)(x)>2^k\}:=\bigcup_{j}Q_j^k,
\]
where $\{Q_j^k\}_j\subset \mathcal S(R)$ is the collection of maximal dyadic cubes in $\Omega_k$ and we denote this collection by $\mathcal Q_k$.
We have the following dyadic maximum principle:
\[
  \sum_{Q\in\mathcal S, Q\supset Q_j^k\atop Q\subset R}\prod_{i=1}^2\frac{1}{|Q|^{1-\alpha/{2n}}}
  \int_Q f_i d\sigma_i >2^k
\]
and
\[
  \sum_{Q\in\mathcal S, Q\supsetneq Q_j^k\atop Q\subset R}\prod_{i=1}^2\frac{1}{|Q|^{1-\alpha/{2n}}}
  \int_Q f_i d\sigma_i \le 2^k.
\]
Let $E(Q_j^k)= Q_j^k \cap \Omega_{k+1}\setminus \Omega_{k+2}$. Then for any $x\in E(Q_j^k)$, we have
\[
  2^{k+1}< \mathcal I_\alpha^{\mathcal S(R)} (f_1\sigma_1, f_2\sigma_2)(x)\le \mathcal I_\alpha^{\mathcal S(R)} (1_{Q_j^k} f_1\sigma_1, 1_{Q_j^k} f_2\sigma_2)(x)+2^k.
\]
Therefore, for any $x\in E(Q_j^k)$,
\begin{equation}\label{eq:local}
\mathcal I_\alpha^{\mathcal S(R)} (1_{Q_j^k} f_1\sigma_1, 1_{Q_j^k} f_2\sigma_2)(x)>2^k.
\end{equation}
Now we have
\begin{eqnarray*}
&&\|\mathcal I_\alpha^{\mathcal S(R)} (f_1\sigma_1, f_2\sigma_2)\|_{L^q(w)}^q\\
&\lesssim& \sum_{k,j}2^{kq}w(E(Q_j^k))\\
&=& \sum_{k,j: w(E(Q_j^k))>\delta w(Q_j^k)}2^{kq}w(E(Q_j^k))+\sum_{k,j: w(E(Q_j^k))\le\delta w(Q_j^k)}2^{kq}w(E(Q_j^k))\\
&\lesssim& \sum_{k,j\atop w(E(Q_j^k))>\delta w(Q_j^k)}w(E(Q_j^k))^{1-q}\bigg(\int_{E(Q_j^k)}\mathcal I_\alpha^{\mathcal S(R)} (1_{Q_j^k} f_1\sigma_1, 1_{Q_j^k} f_2\sigma_2)d w\bigg)^q\\
&&\quad+ \delta \|\mathcal I_\alpha^{\mathcal S(R)} (f_1\sigma_1, f_2\sigma_2)\|_{L^q(w)}^q.
\end{eqnarray*}
By letting $\delta$ be sufficiently small, it suffices to estimate
\begin{equation}\label{eq:key}
\sum_{k,j\atop w(E(Q_j^k))>\delta w(Q_j^k)}w(E(Q_j^k))^{1-q}\bigg(\int_{E(Q_j^k)}\mathcal I_\alpha^{\mathcal S(R)} (1_{Q_j^k} f_1\sigma_1, 1_{Q_j^k} f_2\sigma_2)d w\bigg)^q.
\end{equation}
In the following, we  assume that all $k$ are in the same parity. Without loss of generality, we further assume that all $k$ are even. Then $E(Q_j^k)$ will be pairwise disjoint. Denote
\[
  \mathbb{K}:=\{k: k\, \textup{is even and}\,  w(E(Q_j^k))>\delta w(Q_j^k)\}.
\]
In the rest of this paper, all the sum on $k$ will be understood as on $k\in\mathbb K$. Notice that for $k\in \mathbb K$, $w(E(Q_j^k))\neq 0$, which means that $Q_j^k\notin \mathcal Q_{k+2}$.

\section{Proof of Theorem~\ref{thm:main}: The Strong Type}\label{sec:s2}
\subsection{The special case}
First, we investigate the special case $f_1=1^{}_Q$ and $\supp f_2\subset Q$, where $Q\in \mathcal S$.
 We have the following result.
 \begin{Lemma}\label{lm:special}
Let $\sigma_1,\sigma_2, w$ be positive locally finite Borel measures and $q\ge p_2$. Then for any sparse family $\mathcal S\subset \mathscr D$ and cube $Q\in\mathcal S$,
\begin{equation}\label{eq:e5}
\int_{Q}\mathcal I_\alpha^{\mathcal S}(1^{}_{Q}\sigma_1, 1^{}_{Q} f_2\sigma_2)^{q} dw\lesssim(\mathcal T^{\mathcal S} +\mathcal T_2^{\mathcal S, *})^q \sigma_1(Q)^{q/{p_1}}\|f_2\|_{L^{p_2}(\sigma_2)}^q.
\end{equation}
 \end{Lemma}

 \begin{proof}
Without loss of generality, we can assume that $f_2$ is non-negative. First of all, notice that
\begin{eqnarray*}
&&1_Q I_\alpha^{\mathcal S}(1^{}_{Q}\sigma_1, 1^{}_{Q} f_2\sigma_2)\\
&=& I_\alpha^{\mathcal S(Q)}(1^{}_{Q}\sigma_1, 1^{}_{Q} f_2\sigma_2)+ 1_Q \sum_{\tilde Q\supsetneq Q\atop \tilde Q\in \mathcal S}\frac{\sigma_1(Q)\int_{Q}f_2 d\sigma_i}{|\tilde Q|^{2-\alpha/{n}}} \\
&\lesssim&  I_\alpha^{\mathcal S(Q)}(1^{}_{Q}\sigma_1, 1^{}_{Q} f_2\sigma_2).
\end{eqnarray*}
Therefore, by the previous arguments, it suffices to estimate
\begin{equation}\label{eq:e4}
\sum_{k,j\atop w(E(Q_j^k))>\delta w(Q_j^k)}w(E(Q_j^k))^{1-q}\bigg(\int_{E(Q_j^k)}\mathcal I_\alpha^{\mathcal S(Q)} (1_{Q_j^k} \sigma_1, 1_{Q_j^k} f_2\sigma_2)d w\bigg)^q.
\end{equation}
We have
\begin{eqnarray*}
 &&\sum_{k,j}w(Q_j^k)^{1-q} \left(\int_{E(Q_j^k)}\mathcal I_\alpha^{\mathcal S(Q)} (1_{Q_j^k} \sigma_1, 1_{Q_j^k} f_2\sigma_2)d w\right)^q\\
 &=& \sum_{k,j}w(Q_j^k)^{1-q} \left(\int_{Q_j^k}f_2\mathcal I_\alpha^{\mathcal S(Q)}(1_{Q_j^k} \sigma_1, 1_{E(Q_j^k)}w) d\sigma_2\right)^q\\
 &\lesssim&  \sum_{k,j}w(Q_j^k)^{1-q} \left(\int_{Q_j^k\setminus \Omega_{k+2}}f_2\mathcal I_\alpha^{\mathcal S(Q)}(1_{Q_j^k} \sigma_1, 1_{E(Q_j^k)}w) d\sigma_2\right)^q\\
 &&+ \sum_{k,j}w(Q_j^k)^{1-q} \left(\int_{Q_j^k\cap \Omega_{k+2}}f_2\mathcal I_\alpha^{\mathcal S(Q)}(1_{Q_j^k} \sigma_1, 1_{E(Q_j^k)}w)d\sigma_2\right)^q\\
 &:=&J_1+J_2.
\end{eqnarray*}
First, we estimate $J_1$. We have
\begin{eqnarray*}
J_1&\le&  \sum_{k,j}w(Q_j^k)^{1-q} \left(\int_{Q_j^k\setminus \Omega_{k+2}}\mathcal I_\alpha^{\mathcal S(Q)}(1_{Q_j^k} \sigma_1, 1_{E(Q_j^k)}w)^{p_2'} d\sigma_2\right)^{q/{p_2'}}\\
&&\quad\times \left(\int_{Q_j^k\setminus \Omega_{k+2}}f_2^{p_2}d\sigma_2\right)^{q/{p_2}}\\
&\le&( \mathcal T_2^{\mathcal S, *})^q\sum_{k,j}w(Q_j^k)^{1-q} w(Q_j^k)^{q-1}\sigma_1(Q_j^k)^{q/{p_1}}\\
&&\quad\times
\left(\int_{Q\setminus \Omega_{k+2}}f_2^{p_2}d\sigma_2\right)^{q/{p_2}}\\
&\le&( \mathcal T_2^{\mathcal S, *})^q \sigma_1(Q)^{q/{p_1}}\left(\sum_{k,j}\int_{Q_j^k\setminus \Omega_{k+2}}f_2^{p_2}d\sigma_2\right)^{q/{p_2}}\\
&\le& ( \mathcal T_2^{\mathcal S, *})^q \sigma_1(Q)^{q/{p_1}} \|f_2\|_{L^{p_2}(\sigma_2)}^q.
\end{eqnarray*}
Next we estimate $J_2$. We can write
\begin{eqnarray*}
&&\int_{Q_j^k\cap \Omega_{k+2}}f_2\mathcal I_\alpha^{\mathcal S(Q)}(1_{Q_j^k} \sigma_1, 1_{E(Q_j^k)}w)d\sigma_2\\
&=& \sum_{R\in\mathcal Q_{k+2}\atop R\subset Q_j^k}\int_{R}f_2\mathcal I_\alpha^{\mathcal S(Q)}(1_{Q_j^k} \sigma_1, 1_{E(Q_j^k)}w)d\sigma_2.
\end{eqnarray*}
Notice that for $x\in R$, $\mathcal I_\alpha^{\mathcal S(Q)}(1_{Q_j^k} \sigma_1, 1_{E(Q_j^k)}w)(x)$ is a constant. Define  $ E_Q^\mu f=\mu(Q)^{-1}\int_Q f d\mu$. Then,
\begin{eqnarray*}
&&\int_{Q_j^k\cap \Omega_{k+2}}f_2\mathcal I_\alpha^{\mathcal S(Q)}(1_{Q_j^k} \sigma_1, 1_{E(Q_j^k)}w)d\sigma_2\\
&=& \sum_{R\in\mathcal Q_{k+2}\atop R\subset Q_j^k}E_R^{\sigma_2}f_2\int_{R}\mathcal I_\alpha^{\mathcal S(Q)}(1_{Q_j^k} \sigma_1, 1_{E(Q_j^k)}w)d\sigma_2.
\end{eqnarray*}

To estimate the above sum we need the tool of principal cubes. Since $Q_j^k\subset Q$,
we denote by $\mathcal G_0$ the collection of the maximal cubes in $\cup_{k\in 2\bbZ}\mathcal Q_k$. We define $\mathcal G_n$ inductively. That is,
\begin{eqnarray*}
\mathcal G_{n+1}&=&\bigcup_{G'\in\mathcal G_n}\{G: \,\textup{maximal dyadic subcube of $G'$ such that}\\
&&\qquad E_G^{\sigma_2}f_2> 4E_{G'}^{\sigma_2}f_2\},
\end{eqnarray*}
where the dyadic system in the above is $\cup_{k\in 2\bbZ}\mathcal Q_k$. Then the collection of principal cubes is $\mathcal G=\cup_{n\ge 0}\mathcal G_n$.
By the definition, we immediately have the following
\begin{equation}\label{eq:principal}
\sum_{G\in\mathcal G}(E_G^{\sigma_2}f_2)^{p_2}\sigma_2(G)\lesssim \|M_{\mathscr D}^{\sigma_2} f_2\|_{L^{p_2}(\sigma_2)}^{p_2}\lesssim \|f_2\|_{L^{p_2}(\sigma_2)}^{p_2}.
\end{equation}
Denote by $G(Q)$ the minimal principal cube contains $Q$. We see from the definition that
\[
  E_Q^{\sigma_2}f_2\le 4 E_{G(Q)}^{\sigma_2}f_2.
\]
Note that
\begin{eqnarray*}
 &&\sum_{R\in\mathcal Q_{k+2}\atop R\subset Q_j^k}E_R^{\sigma_2}f_2\int_{R}\mathcal I_\alpha^{\mathcal S(Q)}(1_{Q_j^k} \sigma_1, 1_{E(Q_j^k)}w)d\sigma_2\\
 &=& \sum_{R\subset Q_j^k, G(R)=G(Q_j^k)\atop R\in \mathcal Q_{k+2}}E_{R}^{\sigma_2}f_2 \int_{R}\mathcal I_\alpha^{\mathcal S(Q)}(1_{Q_j^k} \sigma_1, 1_{E(Q_j^k)}w)(x) d\sigma_2\\
 &&\quad+\sum_{R\subset Q_j^k, R\in\mathcal G\atop R\in \mathcal Q_{k+2}}E_{R}^{\sigma_2}f_2 \int_{R}\mathcal I_\alpha^{\mathcal S(Q)}(1_{Q_j^k} \sigma_1, 1_{E(Q_j^k)}w)(x) d\sigma_2.
\end{eqnarray*}
We have
\begin{eqnarray*}
&&\sum_{k,j}w(Q_j^k)^{1-q} \bigg(\sum_{R\in\mathcal Q_{k+2}\atop R\subset Q_j^k}E_R^{\sigma_2}f_2\int_{R}\mathcal I_\alpha^{\mathcal S(Q)}(1_{Q_j^k} \sigma_1, 1_{E(Q_j^k)}w)d\sigma_2\bigg)^q\\
&\lesssim&\sum_{k,j}w(E(Q_j^k))^{1-q} \bigg( \sum_{R\subset Q_j^k, G(R)=G(Q_j^k)\atop R\in \mathcal Q_{k+2}}E_{R}^{\sigma_2}f_2 \\
&&\times\int_{R}\mathcal I_\alpha^{\mathcal S(Q)}(1_{Q_j^k} \sigma_1, 1_{E(Q_j^k)}w)(x) d\sigma_2\bigg)^q\\
&&+\sum_{k,j}w(Q_j^k)^{1-q} \bigg( \sum_{R\subset Q_j^k, R\in\mathcal G\atop R\in \mathcal Q_{k+2}}E_{R}^{\sigma_2}f_2 \int_{R}\mathcal I_\alpha^{\mathcal S(Q)}(1_{Q_j^k} \sigma_1, 1_{E(Q_j^k)}w)(x) d\sigma_2\bigg)^q\\
&:=& J_{21}+J_{22}.
\end{eqnarray*}

For $J_{21}$, we have
\begin{eqnarray*}
J_{21}&\lesssim& \sum_{G\in\mathcal G}\sum_{k,j\atop G(Q_j^k)=G}w(E(Q_j^k))^{1-q} (E_G^{\sigma_2}f_2)^q\\
&&\quad\times
\bigg( \int_{E(Q_j^k)}\mathcal I_\alpha^{\mathcal S(Q)}(1_{Q_j^k}\sigma_1, 1_{Q_j^k} \sigma_2) dw\bigg)^q\\
&\le& \sum_{G\in\mathcal G}\sum_{k,j\atop G(Q_j^k)=G}(E_G^{\sigma_2}f_2)^q
\int_{E(Q_j^k)}\mathcal I_\alpha^{\mathcal S(Q)}(1_{Q_j^k}\sigma_1, 1_{Q_j^k} \sigma_2)^q dw\\
&\le& \sum_{G\in\mathcal G}(E_G^{\sigma_2}f_2)^q \int_{G}\mathcal I_\alpha^{\mathcal S(Q)}(1_{G}\sigma_1, 1_{G} \sigma_2)^q dw\\
&\le& (\mathcal T^{\mathcal S})^q\sum_{G\in\mathcal G}(E_G^{\sigma_2}f_2)^q \sigma_1(G)^{q/{p_1}}\sigma_2(G)^{q/{p_2}}\\
&\le& (\mathcal T^{\mathcal S})^q \sigma_1(Q)^{q/{p_1}}\|f_2\|_{L^{p_2}(\sigma)}^q.
\end{eqnarray*}

And for $J_{22}$, we have
\begin{eqnarray*}
J_{22}&\le&\sum_{k,j}w(Q_j^k)^{1-q} \bigg( \sum_{R\subset Q_j^k, R\in\mathcal G\atop R\in \mathcal Q_{k+2}}(E_{R}^{\sigma_2}f_2)^{p_2}\sigma_2(R)\bigg)^{q/{p_2}} \\
&&\quad\times\bigg( \sum_{R\subset Q_j^k, R\in\mathcal G\atop R\in \mathcal Q_{k+2}}\sigma_2(R)^{-{\frac{p_2'}{p_2}}}\bigg(\int_{R}\mathcal I_\alpha^{\mathcal S(Q)}(
1_{Q_j^k} \sigma_1, 1_{E(Q_j^k)}w) d\sigma_2\bigg)^{p_2'}\bigg)^{\frac{q}{p_2'}}\\
&\lesssim& \sum_{k,j}w(Q_j^k)^{1-q} \bigg( \sum_{R\subset Q_j^k, R\in\mathcal G\atop R\in \mathcal Q_{k+2}}(E_{R}^{\sigma_2}f_2)^{p_2}\sigma_2(R)\bigg)^{q/{p_2}} \\
&&\quad\times\bigg( \int_{ Q_j^k}\mathcal I_\alpha^{\mathcal S(Q)}(1_{Q_j^k} \sigma_1, 1_{E(Q_j^k)}w)^{p_2'} d\sigma_2\bigg)^{q/{p_2'}}\\
&\lesssim&(\mathcal T_2^{\mathcal S, *})^q \sigma_1(Q)^{q/{p_1}} \bigg(  \sum_{k,j}\sum_{R\subset Q_j^k, R\in\mathcal G\atop R\in \mathcal Q_{k+2}}(E_{R}^{\sigma_2}f_2)^{p_2}\sigma_2(R) \bigg)^{q/{p_2}}\\
&\le& (\mathcal T_2^{\mathcal S, *})^q \sigma_1(Q)^{q/{p_1}}  \bigg( \sum_{G\in\mathcal G}(E_{G}^{\sigma_2}f_2)^{p_2}\sigma_2(G)\bigg)^{q/{p_2}}\\
&\lesssim& (\mathcal T_2^{\mathcal S, *})^q \sigma_1(Q)^{q/{p_1}} \|f_2\|_{L^{p_2}(\sigma_2)}^q,
\end{eqnarray*}
where the fact that $k\in\mathbb K$ and therefore any $R$ appears only once is used.
\end{proof}

\subsection{The general case}
In this subsection, we investigate the general case. Again, we can assume that $f_1$ and $f_2$ are non-negative. By the arguments in Section~\ref{sec:s1},
we only need to estimate the following
\[
 \sum_{k,j}w(Q_j^k)^{1-q}\bigg(\int_{E(Q_j^k)}\mathcal I_\alpha^{\mathcal S(R)} (1_{Q_j^k} f_1\sigma_1, 1_{Q_j^k} f_2\sigma_2)d w\bigg)^q.
\]
Since $p_1+p_2\ge p_1p_2$, we have $p_1'\ge p_2$. Hence
\begin{eqnarray*}
&&\int_{E(Q_j^k)}\mathcal I_\alpha^{\mathcal S(R)} (1_{Q_j^k} f_1\sigma_1, 1_{Q_j^k} f_2\sigma_2)d w\\
&\le&
\int_{Q_j^k\setminus \Omega_{k+2}}f_1\mathcal I_\alpha^{\mathcal S(R)}(1_{E(Q_j^k)}w, 1_{Q_j^k} f_2\sigma_2)d\sigma_1\\
&&+ \int_{Q_j^k\cap \Omega_{k+2}}f_1\mathcal I_\alpha^{\mathcal S(R)}(1_{E(Q_j^k)}w, 1_{Q_j^k} f_2\sigma_2) d\sigma_1\\
&\le& \bigg(\int_{Q_j^k\setminus \Omega_{k+2}}\mathcal I_\alpha^{\mathcal S(R)}(1_{E(Q_j^k)}w, 1_{Q_j^k} f_2\sigma_2)^{p_1'} d\sigma_1\bigg)^{\frac{1}{p_1'}}\bigg(\int_{Q_j^k\setminus \Omega_{k+2}}f_1^{p_1}d\sigma_1\bigg)^{\frac{1}{p_1}}\\
&&\quad+ \sum_{\tilde Q\subset Q_j^k\atop \tilde Q\in\mathcal Q_{k+2}}\int_{\tilde Q}f_1\mathcal I_\alpha^{\mathcal S(R)}(1_{E(Q_j^k)}w, 1_{Q_j^k} f_2\sigma_2) d\sigma_1\\
&\lesssim& (\mathcal T_1^{\mathcal S, *}+\mathcal T_2^{\mathcal S, *})w(Q_j^k)^{1/{q'}}\|f_2\|_{L^{p_2}(\sigma_2)}\bigg(\int_{Q_j^k\setminus \Omega_{k+2}}f_1^{p_1}d\sigma_1\bigg)^{1/{p_1}}\\
&&\quad+ \sum_{\tilde Q\subset Q_j^k\atop \tilde Q\in\mathcal Q_{k+2}}E_{\tilde Q}^{\sigma_1}f_1\int_{\tilde Q}\mathcal I_\alpha^{\mathcal S(R)}(1_{E(Q_j^k)}w, 1_{Q_j^k} f_2\sigma_2) d\sigma_1,
\end{eqnarray*}
where we use Lemma~\ref{lm:special} in the last step.
The summation on the first term is easy to estimate. In fact,
\begin{eqnarray*}
&&\sum_{k,j}w(Q_j^k)^{1-q}(\mathcal T_1^{\mathcal S, *}+\mathcal T_2^{\mathcal S, *})^q w(Q_j^k)^{\frac{q}{q'}}\|f_2\|_{L^{p_2}(\sigma_2)}^q\bigg(\int_{Q_j^k\setminus \Omega_{k+2}}f_1^{p_1}d\sigma_1\bigg)^{\frac{q}{p_1}}\\
&\lesssim& (\mathcal T_1^*+\mathcal T_2^*)^q \|f_2\|_{L^{p_2}(\sigma_2)}^q \bigg(\sum_{k,j}\int_{Q_j^k\setminus \Omega_{k+2}}f_1^{p_1}d\sigma_1\bigg)^{\frac{q}{p_1}}\\
&\le& (\mathcal T_1^*+\mathcal T_2^*)^q \|f_1\|_{L^{p_1}(\sigma_1)}^q\|f_2\|_{L^{p_2}(\sigma_2)}^q.
\end{eqnarray*}
It remains to estimate the summation on the second term. Let $\tilde {\mathcal G}$ be the principal cubes associated to $f_1$ and $\sigma_1$. We have
\begin{eqnarray*}
&&\sum_{k,j}w(Q_j^k)^{1-q}\bigg(\sum_{\tilde Q\subset Q_j^k\atop \tilde Q\in\mathcal Q_{k+2}}E_{\tilde Q}^{\sigma_1}f_1\int_{\tilde Q}\mathcal I_\alpha^{\mathcal S(R)}(1_{E(Q_j^k)}w, 1_{Q_j^k} f_2\sigma_2) d\sigma_1\bigg)^q\\
&\lesssim&\sum_{k,j}w(Q_j^k)^{1-q}\bigg(\sum_{\tilde{G}(\tilde Q)=\tilde{G}(Q_j^k)\atop\tilde Q\subset Q_j^k,  \tilde Q\in\mathcal Q_{k+2}}E_{\tilde Q}^{\sigma_1}f_1\int_{\tilde Q}\mathcal I_\alpha^{\mathcal S(R)}(1_{E(Q_j^k)}w, 1_{Q_j^k} f_2\sigma_2)d\sigma_1\bigg)^q\\
&&+\sum_{k,j}w(Q_j^k)^{1-q}\bigg(\sum_{\tilde Q\subset Q_j^k, \tilde Q\in\tilde{\mathcal G}\atop \tilde Q\in\mathcal Q_{k+2}}E_{\tilde Q}^{\sigma_1}f_1\int_{\tilde Q}\mathcal I_\alpha^{\mathcal S(R)}(1_{E(Q_j^k)}w, 1_{Q_j^k} f_2\sigma_2)d\sigma_1\bigg)^q\\
&:=&I_1+I_2.
\end{eqnarray*}
First, we estimate $I_1$. We have
\begin{eqnarray*}
I_1&\lesssim& \sum_{\tilde{G}\in \tilde{\mathcal G}}\sum_{k,j\atop \tilde{G}(Q_j^k)=\tilde{G}} (E_{\tilde G}^{\sigma_1}f_1)^q \int_{E(Q_j^k)}\mathcal I_\alpha^{\mathcal S(R)} (1_{Q_j^k} \sigma_1, 1_{Q_j^k} f_2\sigma_2)^q dw\\
&\le& \sum_{\tilde{G}\in \tilde{\mathcal G}} (E_{\tilde G}^{\sigma_1}f_1)^q \int_{\tilde{G}}\mathcal I_\alpha^{\mathcal S(R)} (1_{\tilde G} \sigma_1, 1_{\tilde G} f_2\sigma_2)^q dw\\
&\lesssim& (\mathcal T^{\mathcal S} +\mathcal T_2^{\mathcal S, *})^q \|f_2\|_{L^{p_2}(\sigma_2)}^q\sum_{\tilde{G}\in \tilde{\mathcal G}} (E_{\tilde G}^{\sigma_1}f_1)^q \sigma_1(\tilde{G})^{q/{p_1}}\quad\mbox{(by Lemma~\ref{lm:special})}\\
&\le& (\mathcal T^{\mathcal S} +\mathcal T_2^{\mathcal S, *})^q \|f_2\|_{L^{p_2}(\sigma_2)}^q \bigg(\sum_{\tilde{G}\in \tilde{\mathcal G}} (E_{\tilde G}^{\sigma_1}f_1)^{p_1} \sigma_1(\tilde{G})\bigg)^{q/{p_1}}\\
&\lesssim& (\mathcal T^{\mathcal S} +\mathcal T_2^{\mathcal S, *})^q \|f_1\|_{L^{p_1}(\sigma_1)}^q \|f_2\|_{L^{p_2}(\sigma_2)}^q.
\end{eqnarray*}
Next we estimate $I_2$. We have
\begin{eqnarray*}
I_2&\le& \sum_{k,j}w(Q_j^k)^{1-q}\bigg(\sum_{\tilde Q\subset Q_j^k, \tilde Q\in\tilde{\mathcal G}\atop \tilde Q\in\mathcal Q_{k+2}}(E_{\tilde Q}^{\sigma_1}f_1)^{p_1}\sigma_1(\tilde Q)\bigg)^{q/{p_1}}\\
&&\quad\times \bigg(\sum_{\tilde Q\subset Q_j^k, \tilde Q\in\tilde{\mathcal G}\atop \tilde Q\in\mathcal Q_{k+2}}\sigma_1(\tilde Q)^{-p_1'/{p_1}}
\bigg(\int_{\tilde Q}\mathcal I_\alpha^{\mathcal S(R)}(1_{E(Q_j^k)}w, 1_{Q_j^k} f_2\sigma_2)(x)d\sigma_1\bigg)^{p_1'}\bigg)^{q/{p_1'}}\\
&\le& \sum_{k,j}w(Q_j^k)^{1-q}\bigg(\sum_{\tilde Q\subset Q_j^k, \tilde Q\in\tilde{\mathcal G}\atop \tilde Q\in\mathcal Q_{k+2}}(E_{\tilde Q}^{\sigma_1}f_1)^{p_1}\sigma_1(\tilde Q)\bigg)^{q/{p_1}}\\
&&\quad\times \bigg(\int_{ Q_j^k}\mathcal I_\alpha^{\mathcal S(R)}(1_{E(Q_j^k)}w, 1_{Q_j^k} f_2\sigma_2)^{p_1'}d\sigma_1\bigg)^{q/{p_1'}}\\
&\le& (\mathcal T_1^{\mathcal S, *}+\mathcal T_2^{\mathcal S, *})^q \|f_2\|_{L^{p_2}(\sigma_2)}^q\sum_{k,j}\bigg(\sum_{\tilde Q\subset Q_j^k, \tilde Q\in\tilde{\mathcal G}\atop \tilde Q\in\mathcal Q_{k+2}}(E_{\tilde Q}^{\sigma_1}f_1)^{p_1}\sigma_1(\tilde Q)\bigg)^{q/{p_1}}\\
&&\hskip 70mm \qquad \mbox{(by Lemma~\ref{lm:special})}\\
&\le& (\mathcal T_1^{\mathcal S, *}+\mathcal T_2^{\mathcal S, *})^q \|f_2\|_{L^{p_2}(\sigma_2)}^q\bigg(\sum_{\tilde{G}\in \tilde{\mathcal G}} (E_{\tilde G}^{\sigma_1}f_1)^{p_1} \sigma_1(\tilde{G})\bigg)^{q/{p_1}}\\
&\le& (\mathcal T_1^{\mathcal S, *}+\mathcal T_2^{\mathcal S, *})^q\|f_1\|_{L^{p_1}(\sigma_1)}^q \|f_2\|_{L^{p_2}(\sigma_2)}^q.
\end{eqnarray*}

\section{Proof of Theorem~\ref{thm:main}: The Weak Type}\label{sec:s3}
In this section, we focus on the weak type inequality \eqref{eq:weak}. Again, we only need to consider \eqref{eq:dweak} and we assume that $f_1$ and $f_2$ are non-negative.
Notice that $q>1$. If \eqref{eq:dweak} holds, we see from  the Kolmogorov inequality that
\begin{eqnarray*}
&&\frac{1}{w(Q)}\int_Q \mathcal I_\alpha^{\mathcal S} (1_Q f_1\sigma_1, 1_Q f_2\sigma_2)dw\\&\lesssim&
\|\mathcal I_\alpha^{\mathcal S} (1_Q f_1\sigma_1, 1_Q f_2\sigma_2)\|_{L^{q,\infty}(Q, w/{w(Q)})}\\
&=& w(Q)^{-1/q}\|\mathcal I_\alpha^{\mathcal S} (1_Q f_1\sigma_1, 1_Q f_2\sigma_2)\|_{L^{q,\infty}(Q, w)}\\
&\le& \mathcal N_{\textup{weak}} w(Q)^{-1/q} \prod_{i=1}^2\|1_Q f_i\|_{L^{p_i}(\sigma_i)},
\end{eqnarray*}
where $Q\in\mathcal S$.
It follows that
\begin{equation}\label{eq:equi}
\int_Q \mathcal I_\alpha^{\mathcal S} (1_Q f_1\sigma_1, 1_Q f_2\sigma_2)dw\lesssim\mathcal N_{\textup{weak}} w(Q)^{1/{q'}} \prod_{i=1}^2\|1_Q f_i\|_{L^{p_i}(\sigma_i)}.
\end{equation}
Now we assume that \eqref{eq:equi} holds for any $Q\in\mathcal S$. For any $R\in\mathscr D$, we have
\begin{eqnarray*}
 \|\mathcal I_\alpha^{\mathcal S(R)} (f_1\sigma_1, f_2\sigma_2)\|_{L^{q,\infty}(w)}^q &\lesssim&\sup_k 2^{(k+1)q}w(\Omega_{k+1}).
\end{eqnarray*}
Denote $F_j^k=Q_j^k\cap \Omega_{k+1}$. By the discussion in Section~\ref{sec:s1},  we have
\[
  \mathcal I_\alpha^{\mathcal S(R)} (1_{Q_j^k}f_1\sigma_1, 1_{Q_j^k} f_2\sigma_2)(x)>2^k,\, x\in F_j^k.
\]
Then
\begin{eqnarray*}
&&2^{(k+1)q}w(\Omega_{k+1})\\
&\le& \sum_j 2^{(k+1)q}w(F_j^k)\\
&=& \sum_{j: w(F_j^k)\ge \delta w(Q_j^k)} 2^{(k+1)q}w(F_j^k)+\sum_{j:w(F_j^k)< \delta w(Q_j^k) } 2^{(k+1)q}w(F_j^k)\\
&\lesssim& \sum_{j: w(F_j^k)\ge \delta w(Q_j^k)}w(F_j^k)^{1-q}\bigg(\int_{F_j^k}\mathcal I_\alpha^{\mathcal S(R)} (1_{Q_j^k}f_1\sigma_1, 1_{Q_j^k} f_2\sigma_2)(x) dw\bigg)^q\\
&&\quad+ \delta \|\mathcal I_\alpha^{\mathcal S(R)} (f_1\sigma_1, f_2\sigma_2)\|_{L^{q,\infty}(w)}^q\\
&\lesssim& \mathcal N_{\textup{weak}}^q  \sum_j\prod_{i=1}^2\|1_{Q_j^k} f_i\|^q_{L^{p_i}(\sigma_i)}+ \delta \|\mathcal I_\alpha^{\mathcal S(R)} (f_1\sigma_1, f_2\sigma_2)\|_{L^{q,\infty}(w)}^q\\
 &\le&  \mathcal N_{\textup{weak}}^q \prod_{i=1}^2 \bigg(\sum_j\|1_{Q_j^k} f_i\|^{p_i}_{L^{p_i}(\sigma_i)} \bigg)^{q/{p_i}}+ \delta \|\mathcal I_\alpha ^{\mathcal S(R)}(f_1\sigma_1, f_2\sigma_2)\|_{L^{q,\infty}(w)}^q\\
 &\lesssim&\mathcal N_{\textup{weak}}^q  \prod_{i=1}^2\|f_i\|_{L^{p_i}(\sigma_i)}^q+ \delta \|\mathcal I_\alpha^{\mathcal S(R)} (f_1\sigma_1, f_2\sigma_2)\|_{L^{q,\infty}(w)}^q.
\end{eqnarray*}
Therefore, by letting $\delta$ be sufficiently small and using the monotone convergence theorem, we get that \eqref{eq:dweak} holds, regardless a constant independent of the weights.

We see from the above arguments that \eqref{eq:dweak} and \eqref{eq:equi} are equivalent. So we only need to give a characterization for \eqref{eq:equi}.
By the duality argument, it is easy to see that \eqref{eq:equi} is equivalent to the following,
\begin{eqnarray}
\bigg(\int_Q \mathcal I_\alpha ^{\mathcal S}(1_Q w, 1_Q f_2\sigma_2)^{p_1'}d\sigma_1\bigg)^{1/{p_1'}}&\lesssim&\mathcal N_{\textup{weak}} w(Q)^{1/{q'}}\|1_Q f_2\|_{L^{p_2}(\sigma_2)};\label{eq:dual}
\\
\bigg(\int_Q \mathcal I_\alpha^{\mathcal S} (1_Q f_1\sigma_1, 1_Q w)^{p_2'}d\sigma_2\bigg)^{1/{p_2'}}&\lesssim&\mathcal N_{\textup{weak}} w(Q)^{1/{q'}}\|1_Q f_1\|_{L^{p_1}(\sigma_1)}.\nonumber
\end{eqnarray}
Therefore, the necessity part follows immediately, i.e., $\mathcal T_1^{\mathcal S, *}, \mathcal T_2^{\mathcal S, *}\lesssim \mathcal N_{\textup{weak}}<\infty$. For the sufficiency part,
since
$p_1'\ge p_2$, we focus on \eqref{eq:dual}.
By Lemma~\ref{lm:special}, we know that $\mathcal N_{\textup{weak}}<\infty$. Moreover,
\[
  \mathcal N_{\textup{weak}}\simeq \mathcal T_1^{\mathcal S, *}+\mathcal T_2^{\mathcal S, *}.
\]

\end{document}